\documentclass[10pt]{article}

\usepackage{latexsym,amsfonts,amsmath,amssymb,cite,mathrsfs,url,amsthm}
\newtheorem{theorem}{Theorem}
\newtheorem{lemma}{Lemma}
\newtheorem{example}{Example}

\newtheorem{algorithm}{Algorithm}
\newtheorem{remark}{Remark}
\newtheorem{proposition}{Proposition}
\newtheorem{definition}{Definition}
\newtheorem{corollary}{Corollary}
\usepackage{hyperref}
\newcommand{\ri}{\mathrm{i}}
\newcommand{\rd}{\, \mathrm{d}}

\newcommand{\bsk}{\boldsymbol{k}}
\newcommand{\bsell}{\boldsymbol{\ell}}
\newcommand{\bszero}{\boldsymbol{0}}
\newcommand{\bsgamma}{\boldsymbol{\gamma}}
\newcommand{\bsq}{\boldsymbol{q}}
\newcommand{\bsx}{\boldsymbol{x}}
\newcommand{\bsy}{\boldsymbol{y}}
\newcommand{\bsz}{\boldsymbol{z}}
\newcommand{\Kor}{\mathrm{Kor}}
\newcommand{\tr}{\mathrm{tr}}
\newcommand{\wal}{\mathrm{wal}}
\newcommand{\wor}{\mathrm{wor}}
\newcommand{\CC}{\mathbb{C}}
\newcommand{\FF}{\mathbb{F}}
\newcommand{\NN}{\mathbb{N}}
\newcommand{\RR}{\mathbb{R}}
\newcommand{\ZZ}{\mathbb{Z}}

\usepackage{color}

\allowdisplaybreaks

\title{Stability of lattice rules and polynomial lattice rules constructed by the component-by-component algorithm}
\author{Josef Dick\thanks{School of Mathematics and Statistics, University of New South Wales, Sydney NSW 2052, Australia ({\tt josef.dick@unsw.edu.au})}, 
Takashi Goda\thanks{School of Engineering, University of Tokyo, 7-3-1 Bunkyo-ku, Hongo, Tokyo 113-8656, Japan ({\tt goda@frcer.t.u-tokyo.ac.jp})}}
\date{\today}

\begin{document}
\maketitle
\begin{abstract}
We study quasi-Monte Carlo (QMC) methods for numerical integration of multivariate functions defined over the high-dimensional unit cube. Lattice rules and polynomial lattice rules, which are special classes of QMC methods, have been intensively studied and the so-called component-by-component (CBC) algorithm has been well-established to construct  rules which achieve the almost optimal rate of convergence with good tractability properties for given smoothness and set of weights. Since the CBC algorithm constructs rules for given smoothness and weights, not much is known when such rules are used for function classes with different smoothness and/or weights.

In this paper we prove that a lattice rule constructed by the CBC algorithm for the weighted Korobov space with given smoothness and weights achieves the almost optimal rate of convergence with good tractability properties for general classes of smoothness and weights which satisfy some summability conditions. Such a stability result also can be shown for polynomial lattice rules in weighted Walsh spaces. We further give bounds on the weighted star discrepancy and discuss the tractability properties for these QMC rules. The results are comparable to those obtained for Halton, Sobol and Niederreiter sequences.


\end{abstract}

%%%%%%%%%%%%%%%%%%%%%%%%%%%%%%%%%%%%%%%%%%%%%%%%%%
%%%%%%%%%%%%%%%%%%%%%%%%%%%%%%%%%%%%%%%%%%%%%%%%%%
\section{Introduction}
We study numerical integration of multivariate functions defined over the $s$-dimensional unit cube $[0,1)^s$. For an integrable function $f\colon [0,1)^s\to \RR$ we denote the integral of $f$ by
\[ I(f)=\int_{[0,1)^s}f(\bsx)\rd \bsx. \]
For an $N$-element set $P\subset [0,1)^s$, we consider approximating $I(f)$ by
\[ Q_P(f)=\frac{1}{N}\sum_{\bsx\in P}f(\bsx). \]
Such an equal-weight quadrature rule is called a quasi-Monte Carlo (QMC) rule.
We refer to \cite{Niedbook,SJbook,DPbook,DKS13,LPbook} for comprehensive information on QMC rules. 

For a Banach space $W$ with norm $\|\cdot\|_W$, the worst-case error by the algorithm $Q_P$ is defined by
\[ e^{\wor}(P;W) := \sup_{\substack{f\in W\\ \|f\|_W\leq 1}}\left| I(f)-Q_P(f) \right|. \]
Our interest is then to construct a good point set $P$ such that the algorithm $Q_P$ makes $e^{\wor}(P;W)$ small for a given $W$.
However, it is often difficult to know whether one point set $P$ constructed for a certain function space works well for different function spaces as well.
In this paper we show some positive results on this question for lattice rules in weighted Korobov spaces and also for polynomial lattice rules in weighted Walsh spaces.

Lattice rules and polynomial lattice rules are defined respectively as follows (here and in what follows, we denote the set of positive integers by $\NN$):
\begin{definition}[lattice rules]\label{def:lat}
For $N\in \NN$, let $\bsz = (z_1, \ldots, z_s) \in \{1,\ldots,N-1\}^s$. An $N$-element lattice point set is given by
\[ P(\bsz) = \left\{ \left( \left\{ \frac{nz_1}{N}\right\}, \ldots, \left\{ \frac{nz_s}{N}\right\} \right) \mid n=0,1,\ldots, N-1\right\}, \]
where $\{x\} = x - \lfloor x \rfloor$ denotes the fractional part of a non-negative real numbers $x$. The resulting QMC algorithm $Q_{P(\bsz)}$ is called a lattice rule with generating vector $\bsz$.
\end{definition}

\begin{definition}[polynomial lattice rules]\label{def:polylat}
For a prime $b$ and $m\in \NN$, let $p \in \FF_b[x]$ be a polynomial of degree $\deg(p) = m$ over the finite field $\FF_b$ of order $b$ and let $\bsq = (q_1, \ldots, q_s) \in (G_m\setminus \{0\})^s$ where we write $G_m = \{g \in \FF_b[x]: \deg(g) < m\}$. A $b^m$-element polynomial lattice point set is given by
\[ P(p,\bsq) = \left\{ \left( \nu_m \left( \frac{n(x)q_1(x)}{p(x)}\right), \ldots, \nu_m \left( \frac{n(x)q_s(x)}{p(x)}\right) \right) \mid n\in G_m\right\}, \]
where we write
\[ \nu_m\left( \sum_{i=w}^{\infty}t_i x^{-i}\right) = \sum_{i=\max(w,1)}^{m}t_i b^{-i}\in [0,1) .\]
The resulting QMC algorithm $Q_{P(p,\bsq)}$ is called a polynomial lattice rule with modulus $p$ and generating vector $\bsq$.
\end{definition}
\noindent To simplify the notation, we hide the dependence of $\bsz$ (resp., $\bsq$) on $N$ (resp., $m$) in the following.

For $s>2$, there is no known explicit construction, free from any computer search or table lookup, of generating vectors $\bsz$ for lattice rules or $\bsq$ for polynomial lattice rules. Instead, the so-called component-by-component (CBC) algorithm, a greedy algorithm which iteratively searches for one component $z_j$ (or $q_j$) with earlier ones $z_1,\ldots,z_{j-1}$ (or, $q_1,\ldots,q_{j-1}$, respectively) kept unchanged, has been well-established, see for instance \cite{SKJ02,SKJ02b,SR02,Kuo03,DKPS05,CKN06,NC06} among many others. 

After the seminal work of Sloan and Wo\'{z}niakowski \cite{SW98}, it has been standard to consider \emph{weighted} function spaces when constructing point sets,
where a set of weight parameters is introduced in the definition of function spaces to play a role in moderating the relative importance of different variables or groups of variables.
It can be shown under some conditions on the weights that the worst-case error bound for good (polynomial) lattice rules depends only polynomially on the dimension $s$, or even, that the error bound is dimension-independent, see for instance \cite{DP05,Kuo03,SKJ02,SKJ02b}. To prove such tractability results for lattice rules or polynomial lattice rules, not only the smoothness parameter of the function space but also a set of weight parameters are required as inputs in the CBC algorithm. In general, it is unknown whether one QMC rule constructed by the CBC algorithm for given smoothness and weights does also work well for different smoothness and weights.

In this paper we prove that a lattice rule constructed by the CBC algorithm for the weighted Korobov space with certain smoothness and weights achieves the almost optimal rate of convergence with good tractability properties for general classes of smoothness and weights which satisfy some summability conditions. The result with respect to the smoothness parameter is well understood and can readily be derived from \cite[Chapter~5]{Niedbook} and Jensen's inequality. We also refer to an argument in \cite[Section~4.4]{LPbook} which considers the CBC algorithm with a different quality criterion independent of smoothness. However, the stability with respect to the weights is much less known, and our present result is new. Moreover, we show a similar result for polynomial lattice rules in weighted Walsh spaces. We also give bounds on the weighted star discrepancy and discuss the tractability properties of lattice rules and polynomial lattice rules.

%%%%%%%%%%%%%%%%%%%%%%%%%%%%%%%%%%%%%%%%%%%%%%%%%%
%%%%%%%%%%%%%%%%%%%%%%%%%%%%%%%%%%%%%%%%%%%%%%%%%%
\section{Stability of lattice rules with respect to changes in smoothness and weights}\label{sec:Lat}
In this section, we study stability of lattice rules in weighted Korobov spaces.

%%%%%%%%%%%%%%%%%%%%%%%%%%%%%%%%%%%%%%%%%%%%%%%%%%
\subsection{Weighted Korobov space}
Let $f: [0,1)^s \to \RR$ be periodic and given by its Fourier series
\[ f(\bsx) = \sum_{\bsk \in \ZZ^s} \hat{f}(\bsk) \exp(2\pi \ri \bsk\cdot \bsx),\]
where the dot $\cdot$ denotes the usual inner product of two vectors and $\hat{f}(\bsk)$ denotes the $\bsk$-th Fourier coefficient defined by
\[ \hat{f}(\bsk) = \int_{[0,1)^s}f(\bsx)\exp(-2\pi \ri \bsk\cdot \bsx)\rd \bsx. \]

We measure the smoothness of periodic functions by a parameter $\alpha>1/2$.
A set of weight parameters $\bsgamma = (\gamma_u)_{u\subset \NN}$ with $\gamma_u\in \RR_{\geq 0}$ is considered to moderate the relative importance of different variables or groups of variables. 
For a non-empty subset $u\subseteq \{1,\ldots,s\}$ and $\bsk_u\in (\ZZ\setminus \{0\})^{|u|}$, we define
\[ r_\alpha(\bsgamma, \bsk_u) := \gamma_u \prod_{j\in u}\frac{1}{|k_j|^{2\alpha}}. \]
Then the weighted Korobov space with smoothness $\alpha$, denoted by $H^{\Kor}_{\alpha,\bsgamma}$, is a reproducing kernel Hilbert space with the reproducing kernel \cite{H98}
\[ K^{\Kor}_{\alpha,\bsgamma}(\bsx,\bsy) = 1+\sum_{\emptyset\neq u \subseteq \{1, \ldots, s\}}\sum_{\bsk_u \in (\ZZ\setminus \{0\})^{|u|}} r_\alpha(\bsgamma, \bsk_u)\exp(2\pi \ri \bsk_u\cdot (\bsx_u-\bsy_u)), \]
where we write $\bsx_u=(x_j)_{j\in u}$ and $\bsy_u=(y_j)_{j\in u}$. The inner product of the space $H^{\Kor}_{\alpha,\bsgamma}$ is given by
\[ \langle f, g\rangle^{\Kor}_{\alpha,\bsgamma} = \hat{f}(\bszero)\hat{g}(\bszero)+\sum_{\emptyset\neq u \subseteq \{1, \ldots, s\}}\sum_{\bsk_u \in (\ZZ\setminus \{0\})^{|u|}}  \frac{\hat{f}(\bsk_u,\bszero)\hat{g}(\bsk_u,\bszero)}{r_\alpha(\bsgamma, \bsk_u)} . \]
Here, for a non-empty subset $u\subseteq \{1,\ldots,s\}$ such that $\gamma_u=0$, we assume that the corresponding inner sum equals 0 and we set $0/0=0$.
The induced norm is then given by
\[ \| f\|^{\Kor}_{\alpha,\bsgamma} = \sqrt{|\hat{f}(\bszero)|^2+\sum_{\emptyset\neq u \subseteq \{1, \ldots, s\}}\sum_{\bsk_u \in (\ZZ\setminus \{0\})^{|u|}} \frac{|\hat{f}(\bsk_u,\bszero)|^2}{r_\alpha(\bsgamma, \bsk_u)}} . \]

%%%%%%%%%%%%%%%%%%%%%%%%%%%%%%%%%%%%%%%%%%%%%%%%%%
\subsection{CBC algorithm for lattice rules}
In order to construct a good lattice rule which works for the weighted Korobov space $H^{\Kor}_{\alpha,\bsgamma}$ with certain $\alpha$ and $\bsgamma$, we consider the worst-case error $e^{\wor}(P(\bsz);H^{\Kor}_{\alpha,\bsgamma})$ as a quality criterion. Since the worst-case error depends only on the generating vector $\bsz$ for fixed $N$, we simply write $e^{\wor}(\bsz;H^{\Kor}_{\alpha,\bsgamma})$.
It follows from the reproducing property of $H^{\Kor}_{\alpha,\bsgamma}$ that we have
\begin{align*}
 \left(e^{\wor}(\bsz; H^{\Kor}_{\alpha,\bsgamma})\right)^2 & = \int_{[0,1)^s}\int_{[0,1)^s}K^{\Kor}_{\alpha,\bsgamma}(\bsx,\bsy)\rd \bsz\rd \bsy \\
&\quad  -\frac{2}{N}\sum_{\bsx\in P(\bsz)}\int_{[0,1)^s}K^{\Kor}_{\alpha,\bsgamma}(\bsx,\bsy)\rd \bsy+\frac{1}{N^2}\sum_{\bsx,\bsy\in P(\bsz)}K^{\Kor}_{\alpha,\bsgamma}(\bsx,\bsy).
\end{align*} 

Define the dual lattice for $P(\bsz)$ by
\[ P^{\perp}(\bsz)= \{\bsk\in \ZZ^s\mid \bsk\cdot \bsz\equiv 0 \pmod N\}. \]
Then the following property is well known.
\begin{lemma}
For $N\in \NN$ and $\bsz\in \{1,\ldots,N-1\}^s$, we have
\[ \frac{1}{N}\sum_{\bsx\in P(\bsz)}\exp(2\pi \ri \bsk\cdot \bsz)=\begin{cases} 1 & \text{if $\bsk\in P^\perp(\bsz)$,} \\ 0 & \text{otherwise.} \end{cases} \]
\end{lemma}
For a non-empty subset $u\subseteq \{1,\ldots,s\}$, define
\[ P^{\perp}_u(\bsz):= \{\bsk_u \in (\ZZ\setminus \{0\})^{|u|}\mid (\bsk_u,\bszero)\in P^{\perp}(\bsz)\}. \]
Then the squared worst-case error is given by
\[ \left(e^{\wor}(\bsz; H^{\Kor}_{\alpha,\bsgamma})\right)^2 = \sum_{\emptyset\neq u \subseteq \{1, \ldots, s\}}\sum_{\bsk_u\in P^{\perp}_u(\bsz)}r_{\alpha}(\bsgamma, \bsk_u) =: P_{\alpha,\bsgamma,N}(\bsz). \]
There is a concise computable form of the criterion $P_{\alpha,\bsgamma,N}(\bsz)$ when $2\alpha$ is a natural number
\[ P_{\alpha,\bsgamma,N}(\bsz) = \frac{1}{N}\sum_{\bsx\in P(\bsz)}\sum_{\emptyset\neq u \subseteq \{1, \ldots, s\}}\gamma_u\prod_{j\in u}\left( \frac{(2\pi)^{2\alpha}}{(-1)^{\alpha+1}(2\alpha)!}B_{2\alpha}(x_j)\right), \]
where $B_{2\alpha}$ is the Bernoulli polynomial of degree $2\alpha$.

Now the CBC algorithm for lattice rules proceeds as follows:
\begin{algorithm}[CBC for lattice rules]\label{alg:cbc_lat}
Let $s,N\in \NN$, $\alpha>1/2$ and $\bsgamma$ be given.
\begin{enumerate}
\item Let $z^*_1=1$ and $\ell=1$.
\item Compute $P_{\alpha,\bsgamma,N}(z^*_1,\ldots,z^*_\ell, z_{\ell+1})$ for all $z_{\ell+1}\in \{1,\ldots,N-1\}$ and let
\[ z^*_{\ell+1}= \arg\min_{z_{\ell+1}}P_{\alpha,\bsgamma,N}(z^*_1,\ldots,z^*_\ell, z_{\ell+1}). \]
\item If $\ell+1<s$, let $\ell=\ell+1$ and go to Step~2.
\end{enumerate}
\end{algorithm}

\begin{remark}\label{rem:cbc}
For special types of weights $\bsgamma$, the necessary computational cost for the CBC algorithm can be made small by using the fast Fourier transform \cite{NC06}.
In the case of product weights, i.e.,
\[ \gamma_u = \prod_{j\in u}\gamma_j\]
for $\gamma_1,\gamma_2,\ldots\in \RR_{\geq 0}$, the set of non-negative real numbers, we only need $O(sN\log N)$ arithmetic operations with $O(N)$ memory. In the case of POD weights, i.e.,
\[ \gamma_u = \Gamma_{|u|}\prod_{j\in u}\gamma_j\]
for $\gamma_1,\gamma_2,\ldots\in \RR_{\geq 0}$ and $\Gamma_1,\Gamma_2,\ldots\in \RR_{\geq 0}$, we need $O(sN\log N+s^2N)$ arithmetic operations with $O(sN)$ memory.
\end{remark}

As shown, for instance, in \cite[Theorem~5.12]{DKS13}, the squared worst-case error for lattice rules constructed by the CBC algorithm can be bounded as follows.
\begin{proposition}\label{prop:cbc_lat}
Let $s,N\in \NN$, $\alpha>1/2$ and $\bsgamma$ be given. The generating vector $\bsz$ constructed by Algorithm~\ref{alg:cbc_lat} satisfies
\[  P_{\alpha,\bsgamma,N}(\bsz) \leq \left(\frac{1}{\varphi(N)} \sum_{\emptyset\neq u \subseteq \{1, \ldots, s\}}\gamma^{\lambda}_u(2\zeta(2\alpha\lambda))^{|u|}\right)^{1/\lambda},\]
for any $1/(2\alpha)<\lambda\leq 1$, where $\zeta$ denotes the Riemann zeta function and $\varphi$ denotes the Euler totient function defined by
\[ \varphi(N) = \sum_{\substack{1\leq n\leq N\\ \mathrm{gcd}(n,N)=1}}1.\]
\end{proposition}

\begin{remark}\label{rem:euler_totient}
Regarding the Euler totient function, a classical work by Rosser and Schoenfeld \cite[Theorem~15]{RS62} shows that we have
\[ \frac{1}{\varphi(N)} < \frac{1}{N}\left(e^c \log\log N+\frac{2.50637}{\log\log N}\right), \]
for any $N\geq 3$, with $c=0.577\ldots$ denoting the Euler's constant. This implies that, for arbitrarily small $\varepsilon>0$, there exists a constant $A_{\varepsilon}$ such that
\[ \frac{1}{\varphi(N)} \leq \frac{A_{\varepsilon}}{N^{1-\varepsilon}} \]
holds where $A_{\varepsilon}\to \infty$ as $\varepsilon\to 0$. Therefore, noting that $P_{\alpha,\bsgamma,N}(\bsz)$ represents the \emph{squared} worst-case error, the lattice rules constructed by Algorithm~\ref{alg:cbc_lat} achieve a convergence rate of the worst-case error arbitrarily close to $O(N^{-\alpha})$. Since we know from \cite{SW01} that it is not possible to achieve a convergence rate better than $O(N^{-\alpha})$, the result in Proposition~\ref{prop:cbc_lat} is almost optimal.
\end{remark}

\begin{remark}\label{rem:jensen}
For any $0<\delta\leq 1$, let us write $\bsgamma^{1/\delta}=(\gamma_u^{1/\delta})_{u\subset \NN}$. Then it follows from Jensen's inequality that
\begin{align*}
\left( P_{\alpha/\delta,\bsgamma^{1/\delta},N}(\bsz)\right)^{\delta} & = \left( \sum_{\emptyset\neq u \subseteq \{1, \ldots, s\}}\sum_{\bsk_u\in P^{\perp}_u(\bsz)}r_{\alpha/\delta}(\bsgamma^{1/\delta}, \bsk_u) \right)^{\delta} \\
& \leq \sum_{\emptyset\neq u \subseteq \{1, \ldots, s\}}\sum_{\bsk_u\in P^{\perp}_u(\bsz)}\left(r_{\alpha/\delta}(\bsgamma^{1/\delta}, \bsk_u)\right)^{\delta} \\
& = \sum_{\emptyset\neq u \subseteq \{1, \ldots, s\}}\sum_{\bsk_u\in P^{\perp}_u(\bsz)}r_{\alpha}(\bsgamma, \bsk_u) = P_{\alpha,\bsgamma,N}(\bsz).
\end{align*}
Thus, the generating vector $\bsz$ constructed by Algorithm~\ref{alg:cbc_lat} based on the criterion $P_{\alpha,\bsgamma,N}(\bsz)$ also works for weighted Korobov spaces with special types of smoothness and weights, i.e.,
\[ \alpha'=\frac{\alpha}{\delta}, \quad \bsgamma'=\bsgamma^{1/\delta}, \]
for any $0<\delta\leq 1$. In the next subsection, we prove a more general stability result than what we obtain simply from Jensen's inequality.
\end{remark}

%%%%%%%%%%%%%%%%%%%%%%%%%%%%%%%%%%%%%%%%%%%%%%%%%%
\subsection{Stability result}

First we note that $P_{\alpha,\bsgamma,N}(\bsz)$ is bounded below by another quality criterion, the Zaremba index or also called figure of merit (\cite[Chapter~5]{Niedbook}, \cite[Chapter~4]{SJbook})
\[ \rho_{\alpha,\bsgamma, N}(\bsz) := \max_{\emptyset\neq u \subseteq \{1, \ldots, s\}}\max_{\bsk_u \in P_u^\perp(\bsz)} r_{\alpha}(\bsgamma, \bsk_u).\] This criterion turns out to be very useful in our context (it has for instance recently also been used in \cite{KKNU19}). 

As one of the main results of this paper, we prove the following upper bound on the squared worst-case error (cf. \cite[Theorem~5.34]{Niedbook}, \cite{NS90,SK83}).
\begin{theorem}\label{thm:lat}
Let $s,N\in \NN$ and $\bsz\in \{1,\ldots,N-1\}^s$. For any $\alpha,\alpha'> 1/2$ and $\bsgamma,\bsgamma'\in \RR_{\geq 0}^{\NN}$ such that $\gamma_v\geq \gamma_u$ whenever $v\subset u$, we have
\[ P_{\alpha',\bsgamma',N}(\bsz) \leq c_{\alpha'}(\rho_{\alpha,\bsgamma, N}(\bsz))^{\alpha'/\alpha}\sum_{\emptyset \neq u \subseteq \{1, \ldots, s\}}\frac{\gamma'_u}{\gamma_u^{\alpha'/\alpha}}\left( \frac{2^{2\alpha'+1}}{2^{2\alpha'-1}-1}\right)^{|u|}(\log_2 N)^{|u|-1}, \]
with 
\[ c_{\alpha'} = (1+\zeta(2\alpha'))+(2^{2\alpha'}+\zeta(2\alpha'))\frac{2^{2\alpha'-1}-1}{2^{4\alpha'}}. \]
\end{theorem}

We defer the proof of this theorem to the end of this section. 

Theorem~\ref{thm:lat} implies that, if we can construct a lattice rule with small $\rho_{\alpha,\bsgamma,N}(\bsz)$ value for given $\alpha$ and $\bsgamma$, 
the same lattice rule also does work for weighted Korobov spaces with different smoothness and weights.
As mentioned in Remark~\ref{rem:jensen}, applying Jensen's inequality leads to a kind of stability result, but it works only for higher smoothness $\alpha'=\alpha/\delta\geq \alpha$ and restrictive form of the weights. 

Let $\bsz$ be the generating vector constructed by the CBC algorithm based on the criterion $P_{\alpha,\bsgamma,N}(\bsz)$ for given $\alpha$ and $\bsgamma$. 
Applying the result from Proposition~\ref{prop:cbc_lat}, we have
\begin{align}\label{eq:bound_cbc_lat}
 P_{\alpha',\bsgamma',N}(\bsz) & \leq c_{\alpha'}\left(\frac{1}{\varphi(N)} \sum_{\emptyset\neq u \subseteq \{1, \ldots, s\}}\gamma^{\lambda}_u(2\zeta(2\alpha\lambda))^{|u|}\right)^{\alpha'/(\alpha \lambda)} \notag \\
 & \quad \times \sum_{\emptyset \neq u \subseteq \{1, \ldots, s\}}\frac{\gamma'_u}{\gamma_u^{\alpha'/\alpha}}\left( \frac{2^{2\alpha'+1}}{2^{2\alpha'-1}-1}\right)^{|u|}(\log_2 N)^{|u|-1}, 
\end{align}
for any $1/(2\alpha)<\lambda\leq 1$. Recalling Remark~\ref{rem:euler_totient}, the bound shown in \eqref{eq:bound_cbc_lat} implies that the lattice rules constructed by Algorithm~\ref{alg:cbc_lat} with $\alpha$ and $\bsgamma$ also achieves a convergence rate of the worst-case error arbitrarily close to $O(N^{-\alpha'})$ for the weighted Korobov space with different parameters $\alpha'$ and $\bsgamma'$, which is almost optimal. Moreover, we can show under some conditions on the weights $\bsgamma$ and $\bsgamma'$ that the worst-case error $P_{\alpha',\bsgamma',N}$ depends only polynomially on the dimension $s$, or even, that the bound is independent of the dimension.

\begin{corollary}\label{cor:lat}
Let $s,N\in \NN$, $\alpha,\alpha'> 1/2$ and $\bsgamma,\bsgamma'\in \RR_{\geq 0}^{\NN}$ such that $\gamma_v\geq \gamma_u$ whenever $v\subset u$. Let $\bsz\in \{1,\ldots,N-1\}^s$ be constructed by Algorithm~\ref{alg:cbc_lat} based on the criterion $P_{\alpha,\bsgamma,N}$. Then the following holds true:
\begin{enumerate}
\item For general weights $\bsgamma$ and $\bsgamma'$, assume that there exist $\lambda, \delta,q,q'\geq 0$ such that $1/(2\alpha)<\lambda<1$, $0 < \delta<\alpha'/(\alpha\lambda)$, 
\[ \sup_{s\in \NN}\frac{1}{s^q}\sum_{\emptyset\neq u \subseteq \{1, \ldots, s\}}\gamma^{\lambda}_u(2\zeta(2\alpha\lambda))^{|u|}<\infty, \]
and
\[ \sup_{s,N\in \NN}\frac{1}{s^{q'} (\varphi(N))^{\delta}} \sum_{\emptyset \neq u \subseteq \{1, \ldots, s\}}\frac{\gamma'_u}{\gamma_u^{\alpha'/\alpha}}\left( \frac{2^{2\alpha'+1}}{2^{2\alpha'-1}-1}\right)^{|u|}(\log_2 N)^{|u|-1} <\infty. \]
Then the worst-case error $P_{\alpha',\bsgamma',N}(\bsz)$ depends only polynomially on $s$ and is bounded by
\[P_{\alpha',\bsgamma',N}(\bsz) \le C_\delta s^{q\alpha'/(\alpha\lambda)+q'} (\varphi(N))^{-\alpha'/(\alpha \lambda)+\delta}, \]
for some constant $C_\delta >0$ which is independent of $s$ and $N$. If the above conditions hold for $q=q'=0$, the worst-case error $P_{\alpha',\bsgamma',N}(\bsz)$ is bounded independently of $s$.
\item In particular, in the case of product weights $\bsgamma$ and $\bsgamma'$, assume that there exists $\lambda\in (1/(2\alpha),1]$ such that
\[ \sum_{j=1}^{\infty}\gamma_j^{\lambda}<\infty\quad \text{and}\quad \sum_{j=1}^{\infty}\frac{\gamma'_j}{\gamma_j^{\alpha'/\alpha}}<\infty. \]
Then the worst-case error $P_{\alpha',\bsgamma',N}(\bsz)$ is independent of $s$ and bounded by 
\[P_{\alpha',\bsgamma',N}(\bsz) \le  C_\delta (\varphi(N))^{-\alpha'/(\alpha \lambda)+\delta}, \]
for arbitrarily small $\delta>0$.
\end{enumerate}
\end{corollary}
\begin{proof}
The result for the first item immediately follows from the bound \eqref{eq:bound_cbc_lat}. The second item can be proven by combining arguments used in \cite[Theorem~4]{Kuo03} and \cite[Lemma~3]{HN03}.
\end{proof}

One of the most important indications from the first item of Corollary~\ref{cor:lat} is that, even for general weights $\bsgamma'$, by choosing $\bsgamma$, such that a fast component-by-component construction is possible (for instance for product weights or POD weights) and the conditions given in Item~1 of Corollary~\ref{cor:lat} are satisfied, then
we can construct a good lattice rule which achieves the almost optimal rate of convergence in $H^{\Kor}_{\alpha',\bsgamma'}$ with good tractability properties. As far as the authors know, such a constructive result for general weights has not been known in the literature. 

To illustrate Corollary~\ref{cor:lat}, we provide some examples of $\alpha,\alpha'$ and $\bsgamma,\bsgamma'$ below, which satisfy the summability conditions in Corollary~\ref{cor:lat}.

\begin{example}(product weights for $\bsgamma$ and general weights for $\bsgamma'$)
Let $\alpha,\alpha'> 1/2$ be arbitrarily given, and $\bsgamma$ be the product weights with $\gamma_j=j^{-2\alpha}$. Then the first summability condition given in Item~1 of Corollary~\ref{cor:lat} is satisfied for any $\lambda\in (1/(2\alpha),1]$ and $q=0$. Now let $\bsgamma'$ be given by
\[ \gamma'_u = \Gamma_{|u|}\left(\frac{1}{|u|}\sum_{j\in u}j^{2\alpha'|u|}\right)^{-1-\varepsilon}, \]
for some $\varepsilon>1/(2\alpha')$ and a sequence $\Gamma_1,\Gamma_2,\ldots\geq 0$. Using an elementary inequality
\[ \frac{1}{|u|}\sum_{j\in u}j^{2\alpha'|u|} \geq \prod_{j\in u}j^{2\alpha'},\]
we have
\[ \frac{\gamma'_u}{\gamma_u^{\alpha'/\alpha}}\leq \Gamma_{|u|}\prod_{j\in u}j^{-2\alpha'\varepsilon}. \]
Let $\Gamma_{|u|}=O(|u|^{\tau})$ or $\Gamma_{|u|}=O(\exp(\tau|u|))$ for some $\tau\geq 0$.
Since we assume $\varepsilon>1/(2\alpha')$, which implies that $\sum_{j=1}^{\infty}j^{-2\alpha'\varepsilon}<\infty$, the argument used in \cite[Lemma~3]{HN03} directly shows that the second summability condition given in Item~1 of Corollary~\ref{cor:lat} is satisfied for arbitrarily small $\delta>0$ and $q'=0$.
\end{example}
\begin{example}(Product weights for both $\bsgamma$ and $\bsgamma'$)
For arbitrarily given $\alpha,\alpha'> 1/2$ and product weights $\bsgamma, \bsgamma'$, let $\gamma_j=j^{-r}$ for some $r>1$. If there exists an arbitrarily small $\varepsilon>1$ such that $\gamma'_j=j^{-\alpha'r/\alpha-\varepsilon}$, then the conditions given in Item~2 of Corollary~\ref{cor:lat} are satisfied for any $\max(1/(2\alpha),1/r)<\lambda\leq 1$.
\end{example}
\begin{example}(POD weights for both $\bsgamma$ and $\bsgamma'$)
For arbitrarily given $\alpha,\alpha'> 1/2$, let 
\[ \gamma_u=\Gamma_{|u|} \prod_{j\in u}\gamma_j\quad \text{and}\quad \gamma'_u=\Gamma'_{|u|} \prod_{j\in u}\gamma'_j \]
for sequences $\gamma_1,\gamma_2,\ldots\geq 0$, $\Gamma_1,\Gamma_2,\ldots\geq 0$, $\gamma'_1,\gamma'_2,\ldots\geq 0$ and $\Gamma'_1,\Gamma'_2,\ldots\geq 0$. Following \cite[Lemma~6.3]{KSS12}, if there exist $1<p<2\alpha$ and $n\in \NN$ such that 
\[ \Gamma_{|u|}=((|u|+n)!)^p \quad \text{and}\quad \sum_{j=1}^{\infty}2\gamma_j^{1/p}\zeta\left(\frac{2\alpha}{p}\right)<1, \]
then the first summability condition given in Item~1 of Corollary~\ref{cor:lat} holds with $q=0$ and $\lambda=1/p$. Moreover, if there exists a constant $C>0$ such that
\[ \frac{\Gamma'_{|u|}}{\Gamma^{\alpha'/\alpha}_{|u|}}\leq C\]
for any non-empty $u\subset \NN$, the second summability condition holds with $q'=0$ and an arbitrarily small $\delta>0$ as long as
\[ \sum_{j=1}^{\infty}\frac{\gamma'_j}{\gamma^{\alpha'/\alpha}_j}<\infty. \]
\end{example}

In the proof of Theorem~\ref{thm:lat}, we shall use the following elementary inequality. We refer to \cite[Lemma~13.24]{DPbook} for the proof.
\begin{lemma}\label{lem:inequ1}
For any real $b>1$ and any $k,t_0\in \NN$, we have
\[ \sum_{t=t_0}^{\infty}b^{-t}\binom{t+k-1}{k-1} \leq b^{-t_0}\binom{t_0+k-1}{k-1}\left( 1-\frac{1}{b}\right)^{-k}. \]
\end{lemma}

\begin{proof}[Proof of Theorem~\ref{thm:lat}]
Recalling the definition of $r_{\alpha'}$, we have
\[ P_{\alpha',\bsgamma',N}(\bsz) = \sum_{\emptyset\neq u \subseteq \{1, \ldots, s\}}\gamma'_u\sum_{\bsk_u\in P^{\perp}_u(\bsz)}\prod_{j\in u}\frac{1}{|k_j|^{2\alpha'}}. \]
Let us define $P^\perp_{u,0}(\bsz) = \{\bsk_u \in \ZZ^{|u|}\setminus \{\bszero\} \mid (\bsk_u,\bszero)\in P^{\perp}(\bsz)\}$, for which we have $P^\perp_u(\bsz)\subseteq P^\perp_{u,0}(\bsz)$.
Moreover, let
\[ \phi_{u,0}(\bsz) := \min_{\bsk_u\in P^\perp_{u,0}(\bsz)}\prod_{j\in u}\max\left(1,|k_j|\right) \]
and
\[ \phi_u(\bsz) := \min_{\bsk_u\in P^\perp_{u}(\bsz)}\prod_{j\in u}|k_j|. \]
Then it is straightforward to see that
\[  \rho_{\alpha,\bsgamma, N}(\bsz) =\max_{\emptyset \neq u \subseteq \{1, \ldots, s\}}\gamma_u\max_{\bsk_u \in P^\perp_u(\bsz)}\prod_{j\in u}\frac{1}{|k_j|^{2\alpha}} = \max_{\emptyset \neq u \subseteq \{1, \ldots, s\}}\frac{\gamma_u}{(\phi_u(\bsz))^{2\alpha}}. \]
Hence it holds for any non-empty subset $u\subseteq \{1,\ldots,s\}$ that
\[ \phi_u(\bsz)\geq \left( \frac{\gamma_u}{\rho_{\alpha,\bsgamma, N}(\bsz)}\right)^{1/(2\alpha)}. \]
Moreover, by assuming $\gamma_v\geq \gamma_u$ whenever $v\subset u$, we obtain a lower bound
\begin{align}\label{eq:lower_bound_mu0}
 \phi_{u,0}(\bsz) = \min_{\emptyset \neq v\subseteq u}\phi_v(\bsz)\geq \min_{\emptyset \neq v\subseteq u}\left( \frac{\gamma_v}{\rho_{\alpha,\bsgamma,N}(\bsz)}\right)^{1/(2\alpha)} = \left( \frac{\gamma_u}{\rho_{\alpha,\bsgamma, N}(\bsz)}\right)^{1/(2\alpha)}.
\end{align}

For a non-empty subset $u\subseteq \{1,\ldots,s\}$, we denote by $\mu_u$ the largest integer such that $2^{\mu_u} < \phi_{u,0}(\bsz)$ holds. It follows from the proof of \cite[Theorem~5.34]{Niedbook} that the inner sum on the expression of $P_{\alpha',\bsgamma',N}(\bsz)$ for a given $u$ with $|u|\geq 2$ is bounded above by
\begin{align*}
\sum_{\boldsymbol{k}_u \in P^\perp_u(\bsz)} \prod_{j\in u}\frac{1}{|k_j|^{2\alpha'}} 
& \leq \sum_{\boldsymbol{k}_u \in P^\perp_{u,0}(\bsz)} \prod_{j\in u}\frac{1}{\max(1,|k_j|)^{2\alpha'}} \\
& \leq \frac{2^{|u|}}{(\phi_{u,0}(\bsz))^{2\alpha'}}\Bigg[ (1+\zeta(2\alpha')) \binom{\mu_u+|u|-1}{|u|-1} \\ & \quad\quad\quad +(2^{2\alpha'}  +\zeta(2\alpha'))\sum_{k=1}^{\infty}2^{(1-2\alpha')k}\binom{k+\mu_u+|u|-2}{|u|-2} \Bigg].
\end{align*}
For the first term in the parenthesis, we have
\[ \binom{\mu_u+|u|-1}{|u|-1}=\prod_{i=1}^{|u|-1}\frac{\mu_u+i}{i}\leq (\mu_u+1)^{|u|-1}.\]
For the second term in the parenthesis, applying Lemma~\ref{lem:inequ1} with $t_0=\mu_u+1, k=|u|-1$ and $b=2^{2\alpha'-1}$ gives
\begin{align*}
\sum_{k=1}^{\infty}2^{(1-2\alpha')k}\binom{k+\mu_u+|u|-2}{|u|-2} & = 2^{(2\alpha'-1)\mu_u}\sum_{k=\mu_u+1}^{\infty}2^{-(2\alpha'-1)k}\binom{k+|u|-2}{|u|-2} \\
& \leq 2^{-(2\alpha'-1)}\binom{\mu_u+|u|-1}{|u|-2}\left( \frac{2^{2\alpha'-1}}{2^{2\alpha'-1}-1}\right)^{|u|-1} \\
& \leq \frac{2^{2\alpha'-1}-1}{2^{4\alpha'-2}}\left( \frac{2^{2\alpha'-1}}{2^{2\alpha'-1}-1}\right)^{|u|}\prod_{i=1}^{|u|-2}\frac{\mu_u+i+1}{i} \\
& \leq \frac{2^{2\alpha'-1}-1}{2^{4\alpha'-2}}\left( \frac{2^{2\alpha'-1}}{2^{2\alpha'-1}-1}\right)^{|u|}(\mu_u+2)^{|u|-2} \\
& \leq \frac{2^{2\alpha'-1}-1}{2^{4\alpha'}}\left( \frac{2^{2\alpha'}}{2^{2\alpha'-1}-1}\right)^{|u|}(\mu_u+1)^{|u|-2}.
\end{align*}

Using these bounds, we obtain
\begin{align*}
& \sum_{\boldsymbol{k}_u \in P^\perp_u(\bsz)} \prod_{j\in u}\frac{1}{|k_j|^{2\alpha'}}  \\
& \leq \frac{2^{|u|}}{(\phi_{u,0}(\bsz))^{2\alpha'}}\Bigg[ (1+\zeta(2\alpha')) (\mu_u+1)^{|u|-1} \\ & \quad \quad \quad +(2^{2\alpha'}+\zeta(2\alpha'))\frac{2^{2\alpha'-1}-1}{2^{4\alpha'}}\left( \frac{2^{2\alpha'}}{2^{2\alpha'-1}-1}\right)^{|u|}(\mu_u+1)^{|u|-2} \Bigg] \\
& \leq c_{\alpha'} \left( \frac{2^{2\alpha'+1}}{2^{2\alpha'-1}-1}\right)^{|u|}\frac{(\mu_u+1)^{|u|-1}}{(\phi_{u,0}(\bsz))^{2\alpha'}}.
\end{align*}
Note that this bound on the inner sum on the expression of $P_{\alpha,\bsgamma,N}(\bsz)$ also applies to the case $|u|=1$.

As shown in \eqref{eq:lower_bound_mu0}, $\phi_{u,0}(\bsz)$ has a lower bound. On the other hand, as proven in \cite[Lemma~5.8]{Niedbook}, $\phi_{u,0}(\bsz)$ also has a trivial upper bound, which is $\phi_{u,0}(\bsz)\leq N/2$. This bound directly means that $\mu_u\leq \log_2 N-1$. Therefore we can bound $P_{\alpha',\bsgamma',N}(\bsz)$ as
\begin{align*}
P_{\alpha',\bsgamma',N}(\bsz) & \leq c_{\alpha'}\sum_{\emptyset \neq u \subseteq \{1, \ldots, s\}}\gamma'_u \left( \frac{2^{2\alpha'+1}}{2^{2\alpha'-1}-1}\right)^{|u|}\frac{(\mu_u+1)^{|u|-1}}{(\phi_{u,0}(\bsz))^{2\alpha'}} \\
& \leq c_{\alpha'}(\rho_{\alpha,\bsgamma,N}(\bsz))^{\alpha'/\alpha}\sum_{\emptyset \neq u \subseteq \{1, \ldots, s\}}\frac{\gamma'_u}{\gamma_u^{\alpha'/\alpha}}\left( \frac{2^{2\alpha'+1}}{2^{2\alpha'-1}-1}\right)^{|u|}(\log_2 N)^{|u|-1}.
\end{align*}
This completes the proof.
\end{proof}

%%%%%%%%%%%%%%%%%%%%%%%%%%%%%%%%%%%%%%%%%%%%%%%%%%
\subsection{Bound on the weighted star discrepancy}
Here we study tractability properties of the weighted star discrepancy for lattice rules constructed by the CBC algorithm based on the criterion $P_{\alpha,\bsgamma,N}$.

For an $N$-element point set $P\subset [0,1)^s$, the local discrepancy function is defined by
\[ \Delta_P(\bsy) := \frac{1}{N}\sum_{\bsx\in P}1_{[\bszero,\bsy)}(\bsx)-\lambda\left( [\bszero,\bsy)\right), \]
for $\bsy\in [0,1)^s$, where $[\bszero,\bsy)=[0,y_1)\times [0,y_2)\times \cdots \times [0,y_s)$ and $1_{[\bszero,\bsy)}$ denotes the characteristic function of the interval $[\bszero,\bsy)$.
For a non-empty subset $u\subseteq \{1,\ldots,s\}$, let us write $P_u=\{\bsx_u=(x_j)_{j\in u}\mid \bsx\in P\}$.
Then the weighted star discrepancy is defined by
\[ D^*_{\bsgamma,N}(P) = \max_{\emptyset \neq u\subseteq \{1,\ldots,s\}}\gamma_u \sup_{\bsy_u\in [0,1)^{|u|}}\left|\Delta_{P_u}(\bsy_u)\right| . \]
In what follows, we focus on lattice point sets and simply write $D^*_{\bsgamma,N}(\bsz)$ instead of $D^*_{\bsgamma,N}(P(\bsz))$.

As shown in \cite[Section~2]{Joe06}, the weighted star discrepancy for a lattice point set with generating vector $\bsz$ is bounded above by
\[ D^*_{\bsgamma,N}(\bsz) \leq \sum_{\emptyset \neq u \subseteq \{1, \ldots, s\}}\gamma_u \left[ 1-\left(1-\frac{1}{N}\right)^{|u|}+\frac{R_{u,N}(\bsz)}{2}\right], \]
where
\[ R_{u,N}(\bsz) = \sum_{\bsk\in P_{u,0}^\perp(\bsz) \cap C^*_{N,|u|}} \prod_{j\in u}\frac{1}{\max(1,|k_j|)},\]
with  
\[ C^*_{N,|u|} = \left\{ \bsk_u\in \ZZ^{|u|}\setminus \{\boldsymbol{0}\} \mid -\frac{N}{2}<k_j\leq \frac{N}{2}, \forall j\in u \right\}.  \]
Moreover, as proven in \cite[Theorem~5.35]{Niedbook}, we have
\[ R_{u,N}(\bsz) \leq \frac{1}{\phi_{u,0}(\bsz)}\left[ \log 2\, (\log_2 N)^{|u|}+3\, (2\log_2 N)^{|u|-1}\right], \]
for any non-empty $u \subseteq \{1, \ldots, s\}$.
By using the lower bound \eqref{eq:lower_bound_mu0} on $\phi_{u,0}(\bsz)$, the following result holds true.
\begin{theorem}
Let $s,N\in \NN$ and $\bsz\in \{1,\ldots,N-1\}^s$. For any $\alpha> 1/2$ and $\bsgamma,\bsgamma'\in \RR_{\geq 0}^{\NN}$ such that $\gamma_v\geq \gamma_u$ whenever $v\subset u$, we have
\begin{align*}
D^*_{\bsgamma',N}(\bsz) \leq & \sum_{\emptyset \neq u \subseteq \{1, \ldots, s\}}\gamma'_u \Bigg[ 1-\left(1-\frac{1}{N}\right)^{|u|}  \\ & \quad\quad +\frac{(\rho_{\alpha,\bsgamma,N}(\bsz))^{1/(2\alpha)}}{2\gamma_u^{1/(2\alpha)}}\left[ \log 2\, (\log_2 N)^{|u|}+3\, (2\log_2 N)^{|u|-1}\right]\Bigg]. 
\end{align*}
\end{theorem}

Applying the result from Proposition~\ref{prop:cbc_lat}, we can prove the following tractability properties.
For general weights $\bsgamma'$ we use \[ 1-\left(1-\frac{1}{N}\right)^{|u|}\leq \frac{|u|}{N}, \]
for any non-empty $u \subseteq \{1, \ldots, s\}$. We assume that the sum $\sum_{j=1}^{\infty}\gamma'_j<\infty$ in the case of product weights to ensure a dimension independent bound on the sum
\[ \sum_{\emptyset \neq u \subseteq \{1, \ldots, s\}}\gamma'_u \left[ 1-\left(1-\frac{1}{N}\right)^{|u|}\right]. \]
We refer to \cite[Lemma~1]{Joe06} for the case of product weights.

\begin{corollary}\label{cor:lat2}
Let $s,N\in \NN$, $\alpha> 1/2$ and $\bsgamma,\bsgamma'\in \RR_{\geq 0}^{\NN}$ such that $\gamma_v\geq \gamma_u$ whenever $v\subset u$. Let $\bsz\in \{1,\ldots,N-1\}^s$ be constructed by Algorithm~\ref{alg:cbc_lat} based on the criterion $P_{\alpha,\bsgamma,N}$. Then the following holds true:
\begin{enumerate}
\item For general weights $\bsgamma$ and $\bsgamma'$, assume that there exist $\lambda, \delta,q,q',q''\geq 0$ such that $1/(2\alpha)<\lambda<1$, $0 < \delta<1/(\alpha\lambda)$, 
\[ \sup_{s\in \NN}\frac{1}{s^q}\sum_{\emptyset\neq u \subseteq \{1, \ldots, s\}}\gamma^{\lambda}_u(2\zeta(2\alpha\lambda))^{|u|}<\infty,\quad \sup_{s\in \NN}\frac{1}{s^{q'}}\sum_{\emptyset\neq u \subseteq \{1, \ldots, s\}}\gamma'_u |u|<\infty, \]
and 
\[ \sup_{s,N\in \NN}\frac{1}{s^{q''}(\varphi(N))^{\delta}}\sum_{\emptyset \neq u \subseteq \{1, \ldots, s\}}\frac{\gamma'_u}{\gamma_u^{1/(2\alpha)}} (2\log_2 N)^{|u|} <\infty. \]
Then the weighted star discrepancy $D^*_{\bsgamma',N}(\bsz)$ depends only polynomially on $s$ and is bounded by  
\[D^*_{\bsgamma',N}(\bsz) \le Cs^{\max(q',q/(2\alpha\lambda)+q'')}(\varphi(N))^{-1/(2\alpha \lambda)+\delta}, \]
for some constant $C>0$ which is independent of $s$ and $N$. If the above conditions hold for $q=q'=q''=0$, then the weighted star discrepancy $D^*_{\bsgamma',N}(\bsz)$ is bounded independently of $s$.
\item In particular, in the case of product weights $\bsgamma$ and $\bsgamma'$, assume that there exists $\lambda\in (1/(2\alpha),1]$ such that
\[ \sum_{j=1}^{\infty}\gamma_j^{\lambda}<\infty \quad \text{and}\quad \sum_{j=1}^{\infty}\frac{\gamma'_j}{\gamma_j^{1/(2\alpha)}}<\infty. \]
Then the weighted star discrepancy $D^*_{\bsgamma',N}(\bsz)$ is bounded independently of $s$ by
\[D^*_{\bsgamma',N}(\bsz) \le C(\varphi(N))^{-1/(2\alpha \lambda)+\delta}, \]
for arbitrarily small $\delta>0$.
\end{enumerate}
\end{corollary}

Let us consider product weights for both $\bsgamma$ and $\bsgamma'$. The first summability condition on $\bsgamma$ is satisfied for any $\lambda\in (1/(2\alpha),1]$ if $\gamma_j=j^{-2\alpha}$. 
Then the second summability condition is given by
\[ \sum_{j=1}^{\infty}j \gamma'_j <\infty. \]
The weighted star discrepancy is bounded independently of $s$ and decays with the almost optimal rate $N^{-1+\delta}$ for arbitrarily small $\delta>0$. 
We note that the summability condition on $\bsgamma'$ is the same as that obtained in \cite{HPT19} (see also \cite{W02,W03}),
where the authors considered explicitly constructed point sets due to Halton, Sobol', and Niederreiter.
It should be pointed out, however, that the latter constructions are extensible in $N$, whereas our lattice point sets are not, so that we need to rerun Algorithm~\ref{alg:cbc_lat} with different values of $N$ based on the same criterion $P_{\alpha,\bsgamma,N}$.

%%%%%%%%%%%%%%%%%%%%%%%%%%%%%%%%%%%%%%%%%%%%%%%%%%
%%%%%%%%%%%%%%%%%%%%%%%%%%%%%%%%%%%%%%%%%%%%%%%%%%
\section{Stability of polynomial lattice rules}\label{sec:PolyLat}
Let us move on to stability of polynomial lattice rules in weighted Walsh spaces.

%%%%%%%%%%%%%%%%%%%%%%%%%%%%%%%%%%%%%%%%%%%%%%%%%%
\subsection{Weighted Walsh spaces}

\begin{definition}
Let $b$ be a prime and $\omega_b:=\exp(2\pi \ri/b)$. For $k\in \NN\cup \{0\}$, we denote the $b$-adic expansion of $k$ by $k=\kappa_0+\kappa_1b+\cdots$, where all except a finite number of $\kappa_i$ are 0.
The $k$-th Walsh function $\wal_k\colon [0,1)\to \CC$ is defined by
\[ \wal_k(x) := \omega_b^{\kappa_0\xi_1+\kappa_1\xi_2+\cdots}, \]
where the $b$-adic expansion of $x\in [0,1)$ is denoted by $x=\xi_1/b+\xi_2/b^2+\cdots$, which is understood to be unique in the sense that infinitely many of the $\xi_i$ are different from $b-1$.

For $s\geq 2$ and $\bsk=(k_1,\ldots,k_s)\in (\NN\cup \{0\})^s$, the $s$-dimensional $\bsk$-th Walsh functions $\wal_{\bsk} \colon [0,1)^s\to \CC$ is defined by
\[ \wal_{\bsk}(\bsx) := \prod_{j=1}^{s}\wal_{k_j}(x_j). \]
\end{definition}

Note that we always use a fixed prime $b$ in the definition of Walsh functions in the rest of this paper. 
The system of Walsh functions is a complete orthogonal system in $L_2([0,1)^s)$. Let $f: [0,1)^s \to \mathbb{R}$ be given by its Walsh series
\[ f(\bsx) = \sum_{\bsk \in (\NN\cup \{0\})^s} \hat{f}(\bsk) \wal_{\bsk}(\bsx), \]
where $\hat{f}(\bsk)$ denotes the $\bsk$-th Walsh coefficient defined by
\[ \hat{f}(\bsk) = \int_{[0,1)^s}f(\bsx)\overline{\wal_{\bsk}(\bsx)} \rd \bsx. \]

Following \cite{DP05}, we measure a smoothness of non-periodic functions by a parameter $\alpha>1/2$.
For $k\in \NN$ with the $b$-adic expansion given by $k=\kappa_0+\kappa_1b+\cdots+\kappa_{a-1}b^{a-1}$ such that $\kappa_{a-1}\neq 0$, let $\mu(k)=a$. 
For a non-empty subset $u\subseteq \{1,\ldots,s\}$ and $\bsk_u\in \NN^{|u|}$, let $\mu(\bsk_u)=\sum_{j\in u}\mu(k_j)$.
Given a set of weights $\bsgamma = (\gamma_u)_{u\subset \NN}$, we define
\[ r_\alpha(\bsgamma, \bsk_u) := \gamma_u b^{-2\alpha \mu(\bsk_u)}. \]
Then the weighted Walsh space with smoothness $\alpha$, denoted by $H^{\wal}_{\alpha,\bsgamma}$, is a reproducing kernel Hilbert space with the reproducing kernel
\[ K^{\wal}_{\alpha,\bsgamma}(\bsx,\bsy) = 1+\sum_{\emptyset\neq u \subseteq \{1, \ldots, s\}}\sum_{\bsk_u \in \NN^{|u|}} r_\alpha(\bsgamma, \bsk_u)\wal_{\bsk_u}(\bsx_u)\overline{\wal_{\bsk_u}(\bsy_u)}, \]
and the inner product
\[ \langle f, g\rangle^{\wal}_{\alpha,\bsgamma} = \hat{f}(\bszero)\hat{g}(\bszero)+\sum_{\emptyset\neq u \subseteq \{1, \ldots, s\}}\sum_{\bsk_u \in \NN^{|u|}}  \frac{\hat{f}(\bsk_u,\bszero)\hat{g}(\bsk_u,\bszero)}{r_\alpha(\bsgamma, \bsk_u)} . \]
Here, for a non-empty subset $u\subseteq \{1,\ldots,s\}$ such that $\gamma_u=0$, we assume that the corresponding inner sum equals 0 and we set $0/0=0$.
The induced norm is then given by
\[ \| f\|^{\wal}_{\alpha,\bsgamma} = \sqrt{|\hat{f}(\bszero)|^2+\sum_{\emptyset\neq u \subseteq \{1, \ldots, s\}}\sum_{\bsk_u \in \NN^{|u|}} \frac{|\hat{f}(\bsk_u,\bszero)|^2}{r_\alpha(\bsgamma, \bsk_u)}} . \]

%%%%%%%%%%%%%%%%%%%%%%%%%%%%%%%%%%%%%%%%%%%%%%%%%%
\subsection{CBC algorithm for polynomial lattice rules}
In order to construct a polynomial lattice rule which works for the weighted Walsh space $H^{\wal}_{\alpha,\bsgamma}$ with certain $\alpha$ and $\bsgamma$, we consider the worst-case error $e^{\wor}(P(p,\bsq);H^{\wal}_{\alpha,\bsgamma})$ as a quality criterion. Since the worst-case error depends only on the generating vector $\bsq$ for a given modulus $p\in \FF_b[x]$ with $\deg(p)=m$, we simply write $e^{\wor}(\bsq;H^{\wal}_{\alpha,\bsgamma})$.

For $k=\kappa_0+\kappa_1b+\cdots \in \NN\cup \{0\}$, let $\tr_m(k):=\kappa_0+\kappa_1x+\cdots+\kappa_{m-1}x^{m-1}\in G_m$. Define the dual lattice for $P(p,\bsq)$ by
\[ P^{\perp}(p,\bsq)= \{\bsk\in (\NN\cup \{0\})^s \mid \tr_m(\bsk)\cdot \bsq \equiv 0 \pmod p\}, \]
where we write $\tr_m(\bsk)\cdot \bsq =\tr_m(k_1)q_1+\cdots+\tr_m(k_s)q_s\in \FF_b[x]$.
The following property was first used in \cite{LT94} and is now well known, see \cite[Lemma~4.75]{DPbook}.
\begin{lemma}
For $m\in \NN$, $p\in \FF_b[x]$ with $\deg(p)=m$ and $\bsq \in (G_m\setminus \{0\})^s$, we have
\[ \frac{1}{b^m}\sum_{\bsx\in P(p,\bsq)}\wal_{\bsk}(\bsx)=\begin{cases} 1 & \text{if $\bsk\in P^\perp(p,\bsq)$,} \\ 0 & \text{otherwise.} \end{cases} \]
\end{lemma}
For a non-empty subset $u\subseteq \{1,\ldots,s\}$, let us write 
\[ P^{\perp}_u(p,\bsq):= \{\bsk_u \in \NN^{|u|}\mid (\bsk_u,\bszero)\in P^{\perp}(p,\bsq)\}. \]
Then the squared worst-case error is simply given by
\[ \left(e^{\wor}(\bsq; H^{\wal}_{\alpha,\bsgamma})\right)^2 = \sum_{\emptyset\neq u \subseteq \{1, \ldots, s\}}\sum_{\bsk_u\in P^{\perp}_u(p,\bsq)}r_{\alpha}(\bsgamma, \bsk_u) =: P_{\alpha,\bsgamma,b^m}(\bsq). \]
There is a concise computable form of the criterion $P_{\alpha,\bsgamma,b^m}(\bsq)$
\[ P_{\alpha,\bsgamma,b^m}(\bsq) = \frac{1}{b^m}\sum_{\bsx\in P(p,\bsq)}\sum_{\emptyset\neq u \subseteq \{1, \ldots, s\}}\gamma_u \prod_{j\in u}\phi_\alpha(x_j), \]
where 
\[ \phi_\alpha(x) = \begin{cases} \frac{b-1}{b^{2\alpha}-b} & \text{for $x=0$,} \\ \frac{b-1}{b^{2\alpha}-b}-\frac{b^{2\alpha}-1}{b^{(2\alpha-1)a}(b^{2\alpha}-b)} & \text{for $\xi_1=\cdots=\xi_{a-1}=0$ and $\xi_a\neq 0$,}\end{cases} \]
see for instance \cite{DP05}. We note that $P_{\alpha,\bsgamma,b^m}(\bsq)$ is bounded below by another quality criterion
\[ \rho_{\alpha,\bsgamma,b^m}(\bsq) := \max_{\emptyset\neq u \subseteq \{1, \ldots, s\}}\max_{\bsk_u \in P_u^\perp(p,\bsq)} r_{\alpha}(\bsgamma, \bsk_u).\]

The CBC algorithm for polynomial lattice rules proceeds as follows:
\begin{algorithm}[CBC for polynomial lattice rules]\label{alg:cbc_polylat}
Let $s,m\in \NN$, $p\in \FF_b[x]$ with $\deg(p)=m$, $\alpha>1/2$ and $\bsgamma$ be given. 
\begin{enumerate}
\item Let $q^*_1=1\in \FF_b[x]$ and $\ell=1$.
\item Compute $P_{\alpha,\bsgamma,b^m}(q^*_1,\ldots,q^*_\ell, q_{\ell+1})$ for all $q_{\ell+1}\in G_m\setminus \{0\}$ and let
\[ q^*_{\ell+1}= \arg\min_{q_{\ell+1}}P_{\alpha,\bsgamma,b^m}(q^*_1,\ldots,q^*_\ell, q_{\ell+1}). \]
\item If $\ell+1<s$, let $\ell=\ell+1$ and go to Step~2.
\end{enumerate}
\end{algorithm}

\begin{remark}\label{rem:cbc2}
Similarly to lattice rules, the necessary computational cost for the CBC algorithm for polynomial lattice rules can be also made significantly small by using the fast Fourier transform.
For instance, in the case of product weights, we only need $O(sN\log N)$ arithmetic operations with $O(N)$ memory.
\end{remark}

As shown in \cite[Theorem~4.4]{DKPS05} for the product-weight cases, the worst-case error for polynomial lattice rules constructed by the CBC algorithm can be bounded as follows.
\begin{proposition}\label{prop:cbc_lat2}
Let $s,m\in \NN$, $\alpha>1/2$ and $\bsgamma$ be given. Let $p\in \FF_b[x]$ be irreducible with $\deg(p)=m$. The generating vector $\bsq$ constructed by Algorithm~\ref{alg:cbc_polylat} satisfies
\begin{align}\label{bound_rho_polylat}
 \rho_{\alpha,\bsgamma,b^m}(\bsq) \leq P_{\alpha,\bsgamma,b^m}(\bsq) \leq \left(\frac{1}{b^m-1} \sum_{\emptyset\neq u \subseteq \{1, \ldots, s\}}\gamma^{\lambda}_u\left( \frac{b-1}{b^{2\alpha\lambda}-b}\right)^{|u|}\right)^{1/\lambda}, 
\end{align}
for any $1/(2\alpha)<\lambda\leq 1$.
\end{proposition}

The rate of convergence of the worst-case error obtained from Proposition~\ref{prop:cbc_lat2} can be arbitrarily close to $O(N^{-\alpha})$, and analogously to Remark~\ref{rem:euler_totient}, this is almost optimal. Also, an argument similar to Remark~\ref{rem:jensen} leads to
\[ \left( P_{\alpha/\delta,\bsgamma^{1/\delta},b^m}(\bsq)\right)^{\delta} \leq P_{\alpha,\bsgamma,b^m}(\bsq) \]
for any $0<\delta\leq 1$.
Thus the generating vector $\bsq$ constructed by Algorithm~\ref{alg:cbc_polylat} based on the criterion $P_{\alpha,\bsgamma,b^m}(\bsq)$ also works for weighted Walsh spaces with special types of smoothness and weights, i.e.,
\[ \alpha'=\frac{\alpha}{\delta}, \quad \bsgamma'=\bsgamma^{1/\delta}. \]
We show more general stability of polynomial lattice rules in the subsequent subsection.

%%%%%%%%%%%%%%%%%%%%%%%%%%%%%%%%%%%%%%%%%%%%%%%%%%
\subsection{Stability result}
As an analogous result for polynomial lattice rules, we prove the following upper bound on the squared worst-case error.
\begin{theorem}\label{thm:polylat}
Let $s,m\in \NN$, $p\in \FF_b[x]$ with $\deg(p)=m$, and $\bsq\in (G_m\setminus \{0\})^s$ be given. For any $\alpha,\alpha'> 1/2$ and $\bsgamma,\bsgamma'\in \RR_{\geq 0}^{\NN}$, we have 
\[ P_{\alpha',\bsgamma',b^m}(\bsq) \leq \left(\rho_{\alpha,\bsgamma,b^m}(\bsq)\right)^{\alpha'/\alpha} \sum_{\emptyset\neq u \subseteq \{1, \ldots, s\}}\frac{\gamma'_u}{\gamma_u^{\alpha'/\alpha}}\left(\frac{b^{2\alpha'-1}(b-1)}{b^{2\alpha'-1}-1}\right)^{|u|} (m+1)^{|u|-1} .\]
\end{theorem}
\begin{proof}
Recalling the definition of $r_{\alpha'}$, we have
\[ P_{\alpha',\bsgamma',b^m}(\bsq) = \sum_{\emptyset\neq u \subseteq \{1, \ldots, s\}}\gamma'_u\sum_{\bsk_u\in P^{\perp}_u(p,\bsq)}b^{-2\alpha' \mu(\bsk_u)}. \]
Let us define
\[ \phi_u(\bsq) := \min_{\bsk_u\in P^\perp_u(p,\bsq)}\mu(\bsk_u). \]
Since $P^\perp_u(p,\bsq)\subseteq \NN^{|u|}$, it is easy to see that $\phi_u(\bsq)\geq |u|$.
Moreover, as we have
\[ \rho_{\alpha,\bsgamma,b^m}(\bsq) = \max_{\emptyset \neq u \subseteq \{1, \ldots, s\}}\gamma_u\max_{\bsk_u\in P^{\perp}_u(p,\bsq)}b^{-2\alpha \mu(\bsk_u)}= \max_{\emptyset \neq u \subseteq \{1, \ldots, s\}}\gamma_ub^{-2\alpha \phi_u(\bsq)},  \]
it holds that
\begin{align}\label{eq:bound_phi}
 \phi_u(\bsq)\geq \frac{1}{2\alpha}\log_b \frac{\gamma_u}{\rho_{\alpha,\bsgamma,b^m}(\bsq)},
\end{align}
for any non-empty subset $u\subseteq \{1,\ldots,s\}$. 

Now it follows from the definition of $\phi_u(\bsq)$ that the inner sum of $P_{\alpha',\bsgamma',b^m}(\bsq)$ over $\bsk_u$ becomes
\begin{align*}
\sum_{\bsk_u \in P^\perp_u(p,\bsq)} b^{-2\alpha'\mu(\boldsymbol{k}_u)} & = \sum_{h=\phi_u(\bsq)}^{\infty}b^{-2\alpha' h} \sum_{\substack{\bsk_u \in P^\perp_u(p,\bsq)\\ \mu(\bsk_u)=h}}1 \\
& = \sum_{h=\phi_u(\bsq)}^{\infty}b^{-2\alpha' h} \sum_{\substack{\bsell_u\in \NN^{|u|}\\ |\bsell_u|_1=h}}\sum_{\substack{\bsk_u \in P^\perp_u(p,\bsq)\\ \mu(k_j)=\ell_j, \forall j\in u}}1.
\end{align*}
Regarding the inner-most sum above, \cite[Lemma~13.8]{DPbook} gives
\[ \sum_{\substack{\bsk_u \in P^\perp_u(p,\bsq)\\ \mu(k_j)=\ell_j, \forall j\in u}}1 \leq \begin{cases}
0 & \text{if $|\bsell_u|_1 < \phi_u(\bsq)$,} \\
(b-1)^{|u|} & \text{if $\phi_u(\bsq)\leq |\bsell_u|_1 < \phi_u(\bsq)+|u|$,} \\
(b-1)^{|u|}b^{|\bsell_u|_1-(\phi_u(\bsq)+|u|-1)} & \text{otherwise.} 
\end{cases} \]
In particular, when $|\bsell_u|_1\geq \phi(\bsq_u)$, we can simplify this bound as
\[ \sum_{\substack{\bsk_u \in P^\perp_u(p,\bsq)\\ \mu(k_j)=\ell_j, \forall j\in u}}1 \leq (b-1)^{|u|}b^{|\bsell_u|_1-\phi_u(\bsq)}. \]
Thus, by applying Lemma~\ref{lem:inequ1} with $t_0=\phi_u(\bsq)-|u|$, $k=|u|$, and $b$ taken to be $b^{2\alpha'-1}$, we have
\begin{align*}
\sum_{\bsk_u \in P^\perp_u(p,\bsq)} b^{-2\alpha'\mu(\boldsymbol{k}_u)} & \leq (b-1)^{|u|}\sum_{h=\phi_u(\bsq)}^{\infty}b^{-2\alpha' h} \sum_{\substack{\bsell_u\in \NN^{|u|}\\ |\bsell_u|_1=h}}b^{h-\phi_u(\bsq)} \\
& =  (b-1)^{|u|}b^{-\phi_u(\bsq)}\sum_{h=\phi_u(\bsq)}^{\infty}b^{-(2\alpha'-1) h}\binom{h-1}{|u|-1} \\
& =  \left(\frac{b-1}{b^{2\alpha'-1}}\right)^{|u|}b^{-\phi_u(\bsq)}\sum_{h=\phi_u(\bsq)-|u|}^{\infty}b^{-(2\alpha'-1) h}\binom{h+|u|-1}{|u|-1} \\
& \leq \left(\frac{b^{2\alpha'-1}(b-1)}{b^{2\alpha'-1}-1}\right)^{|u|} b^{-2\alpha' \phi_u(\bsq)}\binom{\phi_u(\bsq)-1}{|u|-1} \\
& = \left(\frac{b^{2\alpha'-1}(b-1)}{b^{2\alpha'-1}-1}\right)^{|u|} b^{-2\alpha' \phi_u(\bsq)}\prod_{i=1}^{|u|-1}\frac{\phi_u(\bsq)-|u|+i}{i} \\
& \leq \left(\frac{b^{2\alpha'-1}(b-1)}{b^{2\alpha'-1}-1}\right)^{|u|} b^{-2\alpha' \phi_u(\bsq)}(\phi_u(\bsq)-|u|+1)^{|u|-1}.
\end{align*}

Here we show that $\phi_u(\bsq)$ has a trivial upper bound, that is $\phi_u(\bsq)\leq m+|u|$.
Consider the case $k_j=1$ for all $j\in u\setminus \{i\}$ with arbitrarily chosen $i\in u$. Recalling that every $\bsk_u\in P_u^\perp(p,\bsq)$ satisfies
\[ \tr_m(\bsk_u)\cdot \bsq_u=\tr_m(\bsk_{u\setminus \{i\}})\cdot \bsq_{u\setminus \{i\}}+\tr_m(k_i) q_i \equiv 0 \pmod p,\] 
if $\tr_m(\bsk_{u\setminus \{i\}})\cdot \bsq_{u\setminus \{i\}} \not\equiv 0 \pmod p$, then $\tr_m(k_i) q_i\not\equiv 0 \pmod p$. Since $q_i\not\equiv 0 \pmod p$, there exists a unique $k_i\in \{1,\ldots,b^m-1\}$ such that $\tr_m(k_i) q_i\equiv -\tr_m(\bsk_{u\setminus \{i\}})\cdot \bsq_{u\setminus \{i\}}\pmod p$. Thus we have 
\[ \mu(\bsk_u) = \sum_{j\in u\setminus \{i\}}\mu(1)+\mu(k_i) \leq |u|-1+m.\]
Now let us assume $\tr_m(\bsk_{u\setminus \{i\}})\cdot \bsq_{u\setminus \{i\}} \equiv 0 \pmod p$. 
Then we must have $\tr_m(k_i) q_i\equiv 0 \pmod p$. 
Since $k_i> 0$ for any $\bsk_u\in P_u^\perp(p,\bsq)$, $k_i$ is given by $\ell b^m$ for $\ell\in \NN$. 
Taking $k_i=b^m$, we have
\[ \mu(\bsk_u) = \sum_{j\in u\setminus \{i\}}\mu(1)+\mu(b^m) = |u|-1+(m+1)=|u|+m.\]
This argument means that there exists at least one $\bsk_u\in P_u^\perp(p,\bsq)$ such that 
\[ \mu(\bsk_u) \leq |u|+m.\]
which proves our claim
\[ \phi_u(\bsq) = \min_{\bsk_u\in P^\perp_u(p,\bsq)}\mu(\bsk_u) \leq m+|u|. \]

Using this upper bound on $\phi_u(\bsq)$ and the lower bound on $\phi_u(\bsq)$ given in \eqref{eq:bound_phi}, we have
\[ \sum_{\bsk_u \in P^\perp_u(p,\bsq)} b^{-2\alpha'\mu(\boldsymbol{k}_u)}  \leq \left(\rho_{\alpha,\bsgamma,b^m}(\bsq)\right)^{\alpha'/\alpha}\gamma_u^{-\alpha'/\alpha}\left(\frac{b^{2\alpha'-1}(b-1)}{b^{2\alpha'-1}-1}\right)^{|u|}(m+1)^{|u|-1}. \]
Finally we can bound $P_{\alpha',\bsgamma',b^m}(\bsq)$ as
\[ P_{\alpha',\bsgamma',b^m}(\bsq) \leq \left(\rho_{\alpha,\bsgamma,b^m}(\bsq)\right)^{\alpha'/\alpha} \sum_{\emptyset\neq u \subseteq \{1, \ldots, s\}}\frac{\gamma'_u}{\gamma_u^{\alpha'/\alpha}}\left(\frac{b^{2\alpha'-1}(b-1)}{b^{2\alpha'-1}-1}\right)^{|u|} (m+1)^{|u|-1} .\]
This completes the proof.
\end{proof}

Considering the result from Proposition~\ref{prop:cbc_lat2}, we can show tractability results similarly to Corollary~\ref{cor:lat}.
\begin{corollary}\label{cor:polylat}
Let $s,m\in \NN$, $\alpha,\alpha'> 1/2$ and $\bsgamma,\bsgamma'\in \RR_{\geq 0}^{\NN}$. Let $p\in \FF_b[x]$ with $\deg(p)=m$ be irreducible and $\bsq\in (G_m\setminus \{0\})^s$ be constructed by Algorithm~\ref{alg:cbc_polylat} based on the criterion $P_{\alpha,\bsgamma,b^m}$. Then the following holds true:
\begin{enumerate}
\item For general weights $\bsgamma$ and $\bsgamma'$, assume that there exist $\lambda, \delta,q,q'\geq 0$ such that $1/(2\alpha)<\lambda<1$, $0 < \delta<\alpha'/(\alpha\lambda)$, 
\[ \sup_{s\in \NN}\frac{1}{s^q}\sum_{\emptyset\neq u \subseteq \{1, \ldots, s\}}\gamma^{\lambda}_u\left( \frac{b-1}{b^{2\alpha\lambda}-b}\right)^{|u|} <\infty, \]
and
\[ \sup_{s,m\in \NN}\frac{1}{s^{q'}b^{\delta m}}\sum_{\emptyset \neq u \subseteq \{1, \ldots, s\}}\frac{\gamma'_u}{\gamma_u^{\alpha'/\alpha}}\left(\frac{b^{2\alpha'-1}(b-1)}{b^{2\alpha'-1}-1}\right)^{|u|} (m+1)^{|u|-1} <\infty. \]
Then the worst-case error $P_{\alpha',\bsgamma',b^m}(\bsq)$ depends only polynomially on $s$ and is bounded by  
\[P_{\alpha', \bsgamma',b^m}(\bsq) \le Cs^{q\alpha'/(\alpha\lambda)+q'}b^{-(\alpha'/(\alpha \lambda)-\delta)m}, \]
for some constant $C>0$ which is independent of $s$ and $m$. If the above conditions hold for $q=q'=0$, the worst-case error $P_{\alpha',\bsgamma',b^m}(\bsq)$ is bounded independently of $s$.
\item In particular, in the case of product weights $\bsgamma$ and $\bsgamma'$, assume that there exists $\lambda\in (1/(2\alpha),1]$ such that
\[ \sum_{j=1}^{\infty}\gamma_j^{\lambda}<\infty\quad \text{and}\quad \sum_{j=1}^{\infty}\frac{\gamma'_j}{\gamma_j^{\alpha'/\alpha}}<\infty. \]
Then the worst-case error $P_{\alpha',\bsgamma',b^m}(\bsq)$ is independent of $s$ and bounded by
\[ P_{\alpha',\bsgamma',b^m}(\bsq) \le Cb^{-(\alpha'/(\alpha \lambda)-\delta)m}, \]
for arbitrarily small $\delta>0$.
\end{enumerate}
\end{corollary}

%%%%%%%%%%%%%%%%%%%%%%%%%%%%%%%%%%%%%%%%%%%%%%%%%%
\subsection{Bound on weighted star discrepancy}
Finally we study tractability properties of the weighted star discrepancy for polynomial lattice rules constructed by the CBC algorithm based on the criterion $P_{\alpha,\bsgamma, b^m}$.
In what follows, we write $D^*_{\bsgamma, b^m}(\bsq)$ to denote the weighted star discrepancy for a polynomial lattice point set with a given modulus $p$ and generating vector $\bsq$.

For a non-empty subset $u\subseteq \{1,\ldots,s\}$ we write
\[ P^{\perp}_{u,0}(p,\bsq)= \{\bsk_u \in (\NN\cup \{0\})^{|u|}\setminus \{\bszero\} \mid (\bsk_u,\bszero)\in P^{\perp}(p,\bsq)\}, \]
and 
\[ G^*_{m,u}= \left\{ \bsk_u\in (\NN\cup \{0\})^{|u|}\setminus \{\bszero\} \mid k_j<b^m,\forall j\in u\right\}. \]
According to \cite[Corollary~10.16]{DPbook}, the weighted star discrepancy for a polynomial lattice point set is bounded above by
\[ D^*_{\bsgamma, b^m}(\bsq) \leq \sum_{\emptyset \neq u \subseteq \{1, \ldots, s\}}\gamma_u \left[ 1-\left(1-\frac{1}{N}\right)^{|u|}+R_{u, b^m}(\bsq)\right], \]
where
\[ R_{u, b^m}(\bsq) = \sum_{\bsk_u\in P_{u,0}^\perp(p,\bsq) \cap G^*_{m,u}}\prod_{j\in u}\tilde{r}(k_j),\]
with
\[ \tilde{r}(k) = \begin{cases} 1 & \text{if $k=0$,} \\ \frac{1}{b^a \sin(\pi \kappa_{a-1}/b)} & \text{if $k=\kappa_0+\kappa_1 b+\cdots + \kappa_{a-1}b^{a-1}$ with $\kappa_{a-1}\neq 0$.} \end{cases} \]
The weighted star discrepancy is further bounded above as follows.

\begin{theorem}\label{thm:polylat_disc}
Let $s,m\in \NN$, $p\in \FF_b[x]$ with $\deg(p)=m$, and $\bsq\in (G_m\setminus \{0\})^s$ be given. For any $\alpha> 1/2$ and $\bsgamma,\bsgamma'\in \RR_{\geq 0}^{\NN}$ such that $\gamma_v\geq \gamma_u$ whenever $v\subset u$, we have 
\[ D^*_{\bsgamma', b^m}(\bsq) \leq \sum_{\emptyset \neq u \subseteq \{1, \ldots, s\}}\gamma'_u \left[ 1-\left(1-\frac{1}{N}\right)^{|u|}+(b-1)\frac{(\rho_{\alpha,\bsgamma,b^m}(\bsq))^{1/(2\alpha)}}{\gamma_u^{1/(2\alpha)}}(k_b(m+1))^{|u|}\right], \]
where $k_2=1$ and $k_b=1+1/\sin(\pi/b)$ for a prime $b>2$.
\end{theorem}
\begin{proof}
It suffices to give a bound on $R_{u, b^m}(\bsq)$. Let us define
\[ \phi_{u,0}(\bsq) := \min_{\bsk_u\in P^\perp_{u,0}(p,\bsq)}\mu(\bsk_u). \]
By assuming $\gamma_v\geq \gamma_u$ whenever $v\subset u$, it follows from \eqref{eq:bound_phi} that
\[ \phi_{u,0}(\bsq) = \min_{\emptyset \neq v\subseteq u}\phi_{v}(\bsq)\geq \min_{\emptyset \neq v\subseteq u}\frac{1}{2\alpha}\log_b \frac{\gamma_v}{\rho_{\alpha,\bsgamma,b^m}(\bsq)} = \frac{1}{2\alpha}\log_b \frac{\gamma_u}{\rho_{\alpha,\bsgamma,b^m}(\bsq)}. \]
Applying the result from \cite[Theorem~4.34]{Niedbook}, $R_{u, b^m}(\bsq)$ is bounded by
\begin{align*}
R_{u, b^m}(\bsq) & \leq (b-1)\frac{(k_b(m+1))^{|u|}}{b^{\phi_{u,0}(\bsq)}} \\
& \leq (b-1)(k_b(m+1))^{|u|}\left(\frac{\rho_{\alpha,\bsgamma, b^m}(\bsq)}{\gamma_u}\right)^{1/(2\alpha)},
\end{align*}
from which the result immediately follows.
\end{proof}

From this bound on the weighted star discrepancy and Proposition~\ref{prop:cbc_lat2}, we obtain a result similar to Corollary~\ref{cor:lat2}.
\begin{corollary}\label{cor:polylat2}
Let $s,m\in \NN$, $\alpha,\alpha'> 1/2$ and $\bsgamma,\bsgamma'\in \RR_{\geq 0}^{\NN}$ such that $\gamma_v\geq \gamma_u$ whenever $v\subset u$. Let $p\in \FF_b[x]$ with $\deg(p)=m$ be irreducible and $\bsq\in (G_m\setminus \{0\})^s$ be constructed by Algorithm~\ref{alg:cbc_polylat} based on the criterion $P_{\alpha,\bsgamma,b^m}$. Then the following holds true:
\begin{enumerate}
\item For general weights $\bsgamma$ and $\bsgamma'$, assume that there exist $\lambda, \delta,q,q',q''\geq 0$ such that $1/(2\alpha)<\lambda<1$, $0 < \delta<1/(\alpha\lambda)$, 
\[ \sup_{s\in \NN}\frac{1}{s^q}\sum_{\emptyset\neq u \subseteq \{1, \ldots, s\}}\gamma^{\lambda}_u\left( \frac{b-1}{b^{2\alpha\lambda}-b}\right)^{|u|} <\infty, \quad \sup_{s\in \NN}\frac{1}{s^{q'}}\sum_{\emptyset\neq u \subseteq \{1, \ldots, s\}}\gamma'_u |u|<\infty, \]
and 
\[ \sup_{s, m\in \NN}\frac{1}{s^{q''}b^{m \delta}}\sum_{\emptyset \neq u \subseteq \{1, \ldots, s\}}\frac{\gamma'_u}{\gamma_u^{1/(2\alpha)}} (k_b(m+1))^{|u|} <\infty. \]
Then the weighted star discrepancy $D^*_{\bsgamma', b^m}(\bsq)$ depends only polynomially on $s$ and is bounded by 
\[D^*_{\bsgamma', b^m}(\bsq) \le Cs^{\max(q',q/(2\alpha\lambda)+q'')}b^{-m(1/(2\alpha \lambda)-\delta)}, \]
for some constant $C>0$ which is independent of $s$ and $m$. If the above conditions hold for $q=q'=q''=0$, the the weighted star discrepancy $D^*_{\bsgamma', b^m}(\bsq)$ is bounded independently of $s$.
\item In particular, in the case of product weights $\bsgamma$ and $\bsgamma'$, assume that there exists $\lambda\in (1/(2\alpha),1]$ such that
\[ \sum_{j=1}^{\infty}\gamma_j^{\lambda}<\infty \quad \text{and}\quad \sum_{j=1}^{\infty}\frac{\gamma'_j}{\gamma_j^{1/(2\alpha)}}<\infty. \]
Then the weighted star discrepancy $D^*_{\bsgamma', b^m}(\bsq)$ is independent of $s$ and bounded by
\[D^*_{\bsgamma', b^m}(\bsq) \le C b^{-m( 1/(2\alpha \lambda)- \delta )}, \]
for arbitrarily small $\delta>0$.
\end{enumerate}
\end{corollary}

\section*{Acknowledgements}

The authors are grateful to the reviewers for very helpful remarks. Josef Dick is partly supported by the Australian Research Council Discovery Project DP190101197. J.D. would like to thank T.G. for his hospitality during his visit at U. Tokyo and T.G. would like to thank J.D. for his hospitality during his visit at UNSW, where part of this research was carried out.

%%%%%%%%%%%%%%%%%%%%%%%%%%%%%%%%%%%%%%%%%%%%%%%%%%
%%%%%%%%%%%%%%%%%%%%%%%%%%%%%%%%%%%%%%%%%%%%%%%%%%

\end{document}